\documentclass[12pt, letterpaper]{article}

\usepackage[dvips, letterpaper,
headsep=10pt, top=1.4in, right=1.4in, left=1.4in, bottom=1.7in, footskip=35pt]{geometry}

\usepackage[utf8]{inputenc}
\usepackage[T2A]{fontenc}
\usepackage[english,russian]{babel}

\usepackage[intlimits,sumlimits]{amsmath}

\usepackage{amsthm}

\usepackage{amssymb, fancyhdr}

\usepackage{mathrsfs}
\newtheorem{theorem}{Theorem}
\newtheorem{lemma}[theorem]{Lemma}
\newtheorem{tool}[theorem]{Tool}

\newcommand*{\hm}[1]{#1\nobreak\discretionary{}%
{\hbox{$\mathsurround=0pt #1$}}{}} 

\renewcommand{\le}{\leqslant}
\renewcommand{\ge}{\geqslant}

\newenvironment{remark}{\medskip\parindent=0pt\small\textbf{Remark.\ }\ignorespaces}{\medskip\par\normalsize\parindent=18pt}

\newcommand{\eps}{\varepsilon}
\renewcommand{\phi}{\varphi}

\begin{document}\selectlanguage{english}

\pagestyle{fancy}
\fancyhf{}
\chead[\sc Ilya Vinogradov]{\sc\rightmark}
\cfoot{\arabic{page}}
\renewcommand{\footrulewidth}{0,4pt}

\newcounter{sect}
\setcounter{sect}{0}
\newcommand{\mysection}[1]{\addtocounter{sect}{1}\par\vspace{18pt plus 2pt minus 1pt}{\center\textsc{\large\arabic{sect}. \ignorespaces#1}\vspace{6pt plus 2pt minus 1pt}\par}\markright{#1}\setcounter{equation}{0}\setcounter{theorem}{0}}

\numberwithin{equation}{sect}
\numberwithin{theorem}{sect}

\title{A Generalization of a Result of Hardy and Littlewood\footnote{Research carried out in part as the author's undergraduate senior thesis at Columbia University supervised by Prof. Patrick Gallagher.}}

\author{Ilya Vinogradov}

\date{\today}
\maketitle

\begin{abstract}
In this note we study the growth of $$\sum_{m=1}^M\frac1{\|m\alpha\|}$$ as a function of $M$ for different classes of $\alpha\in[0,1)$. Hardy and Littlewood showed in \cite{HL} that for numbers of bounded type, the sum is $\simeq M\log M$. We give a very simple proof for it. Further we show the following for generic $\alpha$. 
For a non-decreasing function $\phi$ tending to infinity, $$\limsup_{M\to\infty}\frac1{\phi(\log M)}\bigg[\frac1{M\log M}\sum_{m=1}^M\frac1{\|m\alpha\|}\bigg]$$ is zero or infinity according as $$\sum\frac1{k\phi(k)}$$ converges or diverges.
\end{abstract}

\mysection{Introduction}

In 1920's and 30's, Hardy and Littlewood made significant progress in the area of Diophantine approximation. They often relied on advanced techniques: deep results from complex analysis, Ces\`aro summability, $L$-functions, etc. Among the simpler tools were continued fractions; main results in this field had been established by Gauss and Legendre and appeared in a complete form, for example, in 
\cite{Chrystal} (first edition published in 1889). In spite of this fact the new approach presented here is based entirely on the theory of continued fractions. We shall show explicitly how the elements of the continued fraction expansion of $\alpha$ govern the behavior of the sum\footnote{In what follows we use the notation $\|x\|=\displaystyle\min_{n\in\mathbf Z}|x-n|$. We use $f=\mathcal O(g)$ and $f\ll g$ interchangeably and if $f\gg g\gg f$, we write $f\simeq g$. }
\begin{equation}
\sum_{m=1}^M\frac1{\|m\alpha\|}.
\end{equation} The primary source of facts on continued fractions is \cite{Khinchin:CF}; we shall quote results from this book and direct the reader to it for proofs. We shall however present proofs of theorems from \cite{Khinchin:CF} for which we have more elementary proofs.

We start by stating the following instrumental

\begin{lemma}Take any irrational $\alpha\in[0,1)$ and an integer $k>1$. Let $\frac{p_k}{q_k}$ be the $k$-th convergent to $\alpha$. Then, we have the inequality
\begin{equation}\frac1{q_k(q_k+q_{k+1})}<\left|\alpha-\frac{p_k}{q_k}\right|<\frac1{q_kq_{k+1}}
\label{basic_eq}.\end{equation} \end{lemma}

This Lemma from \cite{Khinchin:CF} shows how well any given number can be approximated by its convergents. It will be key in estimating our sums as the ``largest'' terms in our sums will typically appear when $m=q_k$ for some $k$.

First we shall tackle the lower bound for the sum $$S_M(\alpha)=\sum_{m=1}^M\frac1{\|m\alpha\|}.$$ Then, we shall derive an upper bound. Whenever we write $S_M(\alpha)$, we imply that it is defined; i.e., $\alpha\not\in\mathbf Q.$

\mysection{Preliminary Results}

Let us introduce some useful terminology and basic tools. A real number is said to be of \emph{bounded type} if the elements of its continued fraction expansion are bounded. It is shown in \cite{Khinchin:CF} and follows from Theorem \ref{Hinchin_th} that the measure of bounded type numbers is zero.

\begin{tool}For real numbers $x,y$, we have $$\|x+y\|\le\|x\|+\|y\|.$$\label{norm}
\end{tool}
\begin{proof}
Let $\|x\|=|x-u|$ and $\|y\|=|y-v|$ for integers $u$ and $v$. Then, by the triangle inequality for real numbers, $$\|x\|+\|y\|=|x-u|+|y-v|\ge|x+y-(u+v)|\ge \min_{n\in\mathbf Z}|x+y-n|=\|x+y\|.$$ \end{proof}

\begin{lemma}For each $m\in\{1,\dots,q_{k+1}\}$, we have \begin{equation}\left\|\frac{mp_k}{q_k}\right\|-\frac1{q_k}<\|m\alpha\|<\left\|\frac{mp_k}{q_k}\right\|+\frac1{q_k}.\label{twosides_eq}\end{equation}  for $k>1$.
\end{lemma}

\begin{proof}We make use of the right-hand side of \eqref{basic_eq}. Multiplying it by a positive integer $m\in\{1,\dots,q_{k+1}\}$ gives $$\left|m\alpha-\frac{mp_k}{q_k}\right|<\frac{m}{q_kq_{k+1}}\le\frac1{q_k}.$$ Since $q_k\ge 2$, we have $\frac m{q_kq_{k+1}}\le\frac12$, and thus
\begin{equation}\label{rhs_eq}\left\|m\alpha-\frac{mp_k}{q_k}\right\|=\left|m\alpha-\frac{mp_k}{q_k}\right|<\frac1{q_k}.\end{equation}
Applying Tool \ref{norm} to the above inequality in two different ways we get the desired result.
\end{proof}

\mysection{Lower Bound on $S_M(\alpha)$}

Now we have enough instruments to begin analyzing the first sum. Theorems \ref{firstsumbelow_th} and \ref{firstsumabove_th} will give respectively lower and upper bounds for the sum.

\begin{theorem}\label{firstsumbelow_th}For any $\alpha\in[0,1)$, fix $k$ so that $q_k\le M<q_{k+1}$. We have $$S_M(\alpha)\gg M\log q_k,$$ with an absolute implied constant.
\end{theorem}

\begin{proof}
We use the right-hand side of \eqref{twosides_eq} to approximate our sum. We shall then use the fact that $(p_k,q_k)=1$ to simplify the bound.

We sum from $m=1+lq_k$ to $(l+1)q_k$ for some $l$ such that $(l+1)q_k\le M$: $$\sum_{m=1+lq_k}^{(l+1)q_k}\!\!\frac1{\|m\alpha\|}\ge\!\!\sum_{m=1+lq_k}^{(l+1)q_k}\!\!\frac1{\|\frac{mp_k}{q_k}\|+\frac1{q_k}}\gg q_k\!\!\sum_{m=1+lq_k}^{(l+1)q_k}\frac1{[mp_k]+1}=q_k\sum_{m=1}^{q_k}\frac1{[mp_k]+1}.$$ The notation $[t]$  means the integer $t'\in\{0,\dots,q_k-1\}$ such that $t\equiv t'\pmod{q_k}.$ As $(p_k,q_k)=1$, the latter expression requires that we sum reciprocals of integers from 1 to $q_k$. The sum on the right becomes $$\sum_{n=1}^{q_k}\frac1n\gg\log q_k.$$ So we have $$\sum_{m=1+lq_k}^{(l+1)q_k}\frac1{\|m\alpha\|}\gg q_k\log q_k.$$ The number of intervals $[1+lq_k,(l+1)q_k]$ over which we summed is $\simeq\frac M{q_k},$ whence $$\frac M{q_k}q_k\log q_k=M\log q_k.$$
\end{proof}

To further analyze the lower bound we invoke the following theorem of Khinchin:

\begin{theorem}[Khinchin]Let $l$ stand for Lebesgue measure on $[0,1)$ and let $a_k\colon[0,1)\to\mathbf N$ denote the $k$-th entry of the continued fraction expansion. For any $\phi\colon\mathbf N\to\mathbf R^+$, \begin{displaymath}l\{a_k>\phi(k)\mbox{ i.o.}\}=1\mbox{ or }0\end{displaymath}according as \begin{displaymath}\sum_k\frac1{\phi(k)}\mbox{ is divergent or convergent.}\end{displaymath} \label{Hinchin_th}
\end{theorem}

One direction in the proof is a direct application of the Borel-Cantelli Lemma as it doesn't rely on independence, while the other requires some more work. Details that exhibit sufficient independence among the elements of continued fractions as well as the proof in its entirety can be found in \cite{Khinchin:CF}.

Next, we discuss two lemmata on metric theory of continued fractions; we don't truly capitalize on them until we get to the upper bound of $S_M(\alpha)$. The proof of the first Lemma can be found in \cite{Sinai:IET,Walters:ET} and in other introductory literature on ergodic theory. The second Lemma is proven in \cite{Khinchin:CF} but we use the proof from \cite{Linnik} which is nicer and elucidates the ergodic nature of the Gauss map.

\begin{lemma}[Invariant measure for the Gauss map]\label{gauss_th}The transformation $$\begin{array}{cccc}T\colon&([0,1),\mathscr B[0,1), l)&\to&([0,1),\mathscr B[0,1),l)\\[5pt]&x&\mapsto&\frac1x-\left[\frac1x\right],\end{array}$$ where $\mathscr B$ stands for the Borel $\sigma$-field and $l$ is Lebesgue measure on it, is ergodic and its invariant measure is $$d\mu(x)=\frac1{\log 2}\cdot\frac{dl(x)}{x+1}.$$
\end{lemma}

\begin{lemma}[Exponential growth of partial quotients]\label{exponential_th}There exist positive constants $a$ and $A$ such that the statement $a<\sqrt[k]{q_k}<A$ holds a.e. on $[0,1)$ for $k$ sufficiently large. \end{lemma}

\begin{remark}The lower bound is actually true for all $k$ and $\alpha$. It suffices to notice that $q_k$ grow at the lowest rate when $\alpha=[1,1,1,\dots]=\frac{\sqrt 5-1}2$. In this case, $q_k$ are consecutive Fibonacci numbers and grow geometrically with common ratio $\frac{\sqrt 5+1}2>1$. The second statement is false on a non-empty negligible set. Take $\alpha=[1,2,2^2,\dots]$. Then, $a_1a_2\dots a_k=2^{\frac{k(k-1)}2}.$ So, $(a_1\dots a_k)^{1/k}=2^{\frac{k-1}2}\to\infty$ as $k\to\infty.$ It follows from the proof below that for this $\alpha$ the sequence $\sqrt[k]{q_k}$ is unbounded, too.

The theorem is probably due to Khinchin, but we present a simpler proof found in \cite{Linnik}. Khinchin strengthened the conclusion to $\sqrt[k]{q_k}\to\gamma$ a.s. in \cite{Khinchin:art}, and L\'evy showed in \cite{Levy} that $\log\gamma=\frac{\pi^2}{12\log 2}.$
\end{remark}

\begin{proof}Let $T$ and $\mu$ be as in Lemma \ref{gauss_th}. Then for any $\alpha\in [0,1 ) $ we have that $a_k(\alpha)=\left[\frac1{T^{k-1}\alpha}\right].$ It is plain that $$a_kq_{k-1}<a_kq_{k-1}+q_{k-2}=q_k<a_kq_{k-1}+q_{k-1}\le2a_kq_{k-1}.$$ Proceeding inductively we conclude that $a_1\dots a_k<q_k<2^ka_1\dots a_k.$ Hence it suffices to show that $$\sqrt[k]{\textstyle\prod_{i=1}^{k}\limits a_i}\to \mbox{const}>0\quad\mbox{a.s.}$$

Here we invoke Lemma \ref{gauss_th}. Take $f(x)=\log\left[\frac1x\right]\in L^1$. By the Pointwise Ergodic Theorem, $$\lim_{k\to\infty}\frac1k\sum_{i=0}^{k-1} \log\left[\frac1{T^{i}\alpha}\right]=\frac1{\log 2}\int_0^1\log\left[\frac1x\right]\frac{dx}{1+x}=C, \quad\mbox{a.s., rate depending on }\alpha$$ It follows by exponentiation that $$e^C=\lim_{k\to\infty}(a_1\dots a_k)^{1/k},\quad\mbox{a.s.}$$ and our proof is complete.
\end{proof}

We now apply these facts to the case at hand. For numbers of bounded type, we have $M<q_{k+1}=a_{k+1}q_k+q_{k-1}<Cq_k$ and thus we can claim that $$S_M(\alpha)\gg M\log M$$ with a universal implied constant. Almost surely we also get the $M\log M$ lower bound. By Theorem \ref{Hinchin_th}, the set $\{a_{k+1}\le q_k \mbox{ eventually}\}$ is of full measure. 
Indeed, $q_k$ grow at least geometrically everywhere and hence the complement $\{a_{k+1}> q_k \mbox{ i.o.}\}$ is null. 
Thus, on the event $\{a_{k+1}\le q_k \mbox{ eventually}\}$, 
$$\log M\le\log (a_{k+1}q_k+q_{k-1})\ll\log q_k^2\ll \log q_k$$ and up to a constant, $\log M$ and $\log q_k$ are the same. It is quite possible that $S_M(\alpha)\gg M\log M$ for all $\alpha$, but we don't have a proof or a counterexample.

\mysection{Upper Bound on $S_M(\alpha)$}

Now we proceed to obtain the upper bound. This result will not be uniform in $\alpha$: specifically, if $\alpha$ is of unbounded type, then our conclusions are weaker.

\begin{theorem}\label{firstsumabove_th}Let the continued fraction expansion of $\alpha\in[0,1)$ be $[a_1,a_2,\dots]$. Given an integer $M$, take $k$ so that $q_k\le M<q_{k+1}$. Then we have $$S_M(\alpha)\ll M\log q_k+a_{k+1}M.$$\end{theorem}

\begin{proof}Let's begin by applying the left inequality of \eqref{twosides_eq}. There are three cases when we don't gain any information from this inequality: when $[mp_k]=0,\pm 1$. We exclude these for now as they require a different treatment. Summing over the remaining $m$ from 1 to $q_k$ we get $$\sum_{\substack{1\le m\le q_k\\{}[mp_k]\ne0,\pm1}}\frac1{\|m\alpha\|}\le \sum_{\substack{1\le m\le q_k\\{}[mp_k]\ne0,\pm1}}\frac1{\|\frac{mp_k}{q_k}\|-\frac1{q_k}}\ll q_k\sum_{\substack{1\le m\le q_k\\{}[mp_k]\ne0,\pm1}}\frac1{[mp_k]-1}.$$

As in the proof of Theorem \ref{firstsumbelow_th}, we notice that the factor of $p_k$ is superfluous as $(p_k,q_k)=1$. Thus, the sum can rewritten as $$q_k\sum_{\substack{1\le n\le q_k\\{}[n]\ne0,\pm1}}\frac1{n-1}.$$ Clearly, this quantity is asymptotic to $q_k\log q_k$. There will be $\simeq\frac M{q_k}$ such terms, inasmuch as we get the first term in the bound.

Now we deal with terms $mp_k\equiv0,\pm1$. We shall estimate the error incurred by replacing $\|m\alpha\|$ by $\|q_k\alpha\|$. I claim that $\|m\alpha\|\hm\ge\frac12\|q_k\alpha\|$ for any $m$ such that $q_k\le m<q_{k+1}.$ This is certainly true if $m=q_k$, so suppose $m$ is \emph{not} a partial quotient of $\alpha.$ Proceed by contradiction. Assume that $\|m\alpha\|<\frac12\|q_k\alpha\|.$ By \eqref{basic_eq}, it follows that $\|q_k\alpha\|<\frac1{q_{k+1}}.$ Thus,
$$\|m\alpha\|<\frac12\|q_k\alpha\|<\frac1{2q_{k+1}}<\frac1{2m}.$$ By Theorem 19 in \cite{Khinchin:CF}, we get that $m$ is a partial quotient of $\alpha$, which contradicts our assumption.

Hence at the expense of a factor of two, we can replace $\|m\alpha\|$ by $\|q_{k+1}\alpha\|$. Thus, the bound for the sum \begin{equation}\label{eq:notcareful}\sum_{[mp_k]=0,\pm1}\frac1{\|m\alpha\|}\end{equation} is $\frac M{q_k}\frac1{\|q_k\alpha\|}\ll a_{k+1}M$. Indeed, from the left-hand side of \eqref{basic_eq} it follows that $\|q_k\alpha\|\gg\frac1{q_{k+1}}$, and desired result follows at once.

\end{proof}

One may wonder whether it is possible to improve on the second term or to get rid of it altogether. Indeed, we have nonchalantly replaced the sum by the product of the number of terms and the largest term. In general, the answer is no. If the sum consists of but one term, little can be done to improve the approximation $\|q_k\alpha\|\approx\frac1{q_{k+1}}$ as can be seen from \eqref{basic_eq}. On the other hand, if the sum consists of many terms (which is the same as saying that $M$ is close to the upper end of the interval), we can improve the bound to $M(1+\log a_{k+1})$. This improvement is based entirely on careful analysis of \eqref{eq:notcareful} and is left to the reader.

We can now discuss some of the consequences of this Theorem. First of all, we have established a result of \cite{HL} using a shorter and more elementary method. Clearly, for numbers of bounded type we have the upper bound $M\log M$, and the implied constant depends on $\alpha$ in both cases. It is curious that in this case the main contribution comes from the ``bulk'' terms (with $\frac1{\|m\alpha\|}$ small), while the ``special'' terms contribute less.

\mysection{Growth Criterion}

We have an upper bound and a lower bound for $S_M(\alpha)$. What is the function that captures the exact growth rate of the sum? The following Theorem lets us decide whether the sum grows faster or slower than any given function. 

\begin{theorem}\label{cases_th}Let $\phi(x)$ be a positive non-decreasing function. Then, for almost every $\alpha\in[0,1)$ we have $$\limsup_{M\to\infty}\left.\frac{S_M(\alpha)}{M\log M}\right/\phi(\log M)=\begin{cases} 0&\mbox{if }\displaystyle\sum_k\frac1{k\phi(k)}<\infty\\\\ \infty&\mbox{if }\displaystyle\sum_k\frac1{k\phi(k)}=\infty.\end{cases}$$
\end{theorem}

\begin{proof}
Let's verify the first line. By Theorem \ref{firstsumabove_th} we have $$\frac{S_M(\alpha)}{M\log M\phi(\log M)}\ll \frac1{\phi(\log M)}+\frac{a_{k+1}}{\log M\phi(\log M)}\ll \frac1{\phi(k\log a)}+\frac{a_{k+1}}{k\phi(k\log a)}$$ where $a$ comes from Theorem \ref{exponential_th}. Now by Khinchin's Theorem, $$l\{a_{k+1}>k\phi(k\log a)\big(\eps-\textstyle\frac1{\phi(k\log a)}\big)\mbox{ i.o.}\}=0$$ for any positive $\eps$ since $$\sum^\infty\frac1{k\phi(k\log a)(\eps-\frac1{\phi(k\log a)})}\simeq\sum^\infty\frac1{k\phi(k\log a)}\simeq\int^\infty\frac{dx}{x\phi(x)}<\infty.$$ We remark here that the sum extends over those $k$ for which $\eps-\frac1{\phi(k\log a)}>0$ and that $\log a>0$ by direct computation. Hence the quantity $$\left.\frac{S_M(\alpha)}{M\log M}\right/\phi(\log M)$$ exceeds the value $\eps$ finitely often with probability one. Taking $\eps=\frac1n$ for $n\in\mathbf N$ we get a countable union of measure zero sets, which itself has measure zero, inasmuch as the limit superior vanishes almost surely, as advertized.  

We go on to prove the second line. We shall concentrate on the term $m=q_k$ and exhibit a suitable sequence $M_k$ so that this term of the sum alone will diverge as $M\to\infty$. It is natural to take $M_k=q_k$. Then, we need to estimate $$\frac1{\|q_k\alpha\|q_k\log q_k\phi(\log q_k)}.$$ Using \eqref{basic_eq}, we get using Lemma \ref{exponential_th} that it is greater than $$\frac{q_{k+1}}{q_k\log q_k\phi(\log q_k)}\gg\frac{a_{k+1}}{\log q_k\phi(\log q_k)}\gg\frac{a_{k+1}}{k\phi(k\log A)}$$  for $k$ sufficiently large. Thus, by Khinchin's Theorem, we need to ensure that $\sum\frac1{k\phi(k\log A)}=\infty$. This is indeed the case since $\phi$ is non-decreasing and the sum can be compared to the integral. Hence, the quantity in question is almost surely unbounded.
\end{proof}

\mysection{Conclusion}

The above discussion summarizes completely the behavior of $S_M(\alpha)$ for all $M$. 
Theorem \ref{cases_th} gives the exact growth of the sum for almost all $\alpha$.  
 
We have obtained results for some measure zero sets, too. For numbers of bounded type, $$\frac{S_M(\alpha)}{M\log M}\simeq 1$$ and the absolute constants can be made explicit if necessary. Hardy and Littlewood claimed that $$\sum_{m=1}^M\frac1{|\sin \pi m\alpha|}=\mathcal O(M\log M)$$ and that this bound was best possible. 

It is worth pointing out that all the results above can be easily adapted to the sum $$\sum_{m=1}^M\frac1{\|m\alpha\|^\beta}$$ for $\beta>1$. Other expressions can be allowed in the sum, too. However, summation may become difficult as the simplification that we have used in Theorem \ref{firstsumbelow_th} might not work. For example, $$\sum_{m=1}^M\frac1{m\|m\alpha\|}$$ is harder to estimate, and more general summations like this one are a direction of further research.   

To better understand the growth of the original quantity $S_M(\alpha)$, it may be advantageous to look at averages of the sum: $$\frac1N\sum_{n=1}^N S_n(\alpha).$$ This is so because we have seen that for certain values of $M$ the sum is unusually large, while for others it is quite small. The general behavior could be elucidated through Ces\`aro means of this kind. This is possible direction for future investigation. 

\mysection{Acknowledgements}

I thank Prof. Sinai for providing me with the original problem and Prof. Gallagher for guiding me in solving that problem and for his support in generalizing it. I am also indebted to many of my friends for helping me with the research, improving the text, and typesetting it. Among them are 
Thomas Dumitrescu,
Dmytro Karabash, 
Alex Kontorovich,
Claire Lackner, 
Jonathan Luk, and 
Daniel Minsky.

\bibliographystyle{plain}
\bibliography{hardyarticle}

\end{document}